\documentclass[12pt,a4paper]{amsart}
\usepackage{graphics}
\usepackage{epsfig}
\usepackage{graphicx}
\usepackage[all]{xy}
\theoremstyle{plain}
\usepackage{amssymb}
\usepackage{enumerate}

\advance\hoffset-20mm \advance\textwidth41mm

\newtheorem{theorem}{Theorem}
\newtheorem{lemma}{Lemma}
\newtheorem*{theo*}{Theorem}
\newtheorem{proposition}{Proposition}
\newtheorem{corollary}{Corollary}
\theoremstyle{definition}
\newtheorem{definition}{Definition}
\newtheorem*{definition*}{Definition}
\newtheorem{example}{Example}
\newtheorem{remark}{Remark}

\def\AA{{\mathbb A}}

\def\PP{{\mathbb P}}
\def\QQ{{\mathbb Q}}
\def\KK{{\mathbb K}}
\def\ZZ{{\mathbb Z}}
\def\XX{{\mathbb X}}
\def\GG{{\mathbb G}}

\newcommand{\cG}{{\ensuremath{\mathcal{G}}}}
\newcommand{\cH}{{\ensuremath{\mathcal{H}}}}

\def\Of{{\mathcal{O}}}

\def\SAut{\mathop{\rm SAut}}
\def\div{\mathop{\rm div}}
\def\reg{\mathop{\rm reg}}

\def\Ker{\mathop{\rm Ker}}
\def\Spec{\mathop{\rm Spec}}

\def\Aut{\mathop{\rm Aut}}

\def\WDiv{\mathop{\rm WDiv}}

\def\Cl{\mathop{\rm Cl}}
\DeclareMathOperator{\Pic}{Pic}
\DeclareMathOperator{\tail}{tail}

\newcommand{\D}{{\mathcal{D}}}
\newcommand{\fan}{\mathcal{S}}

\DeclareMathOperator{\cone}{cone}

\begin{document}
\sloppy
\title[Infinite transitivity on universal torsors]
{Infinite transitivity on universal torsors}
\author{Ivan Arzhantsev}
\thanks{The first and the second authors were partially supported by the Ministry of Education
and Science of Russian Federation, project 8214, by RFBR grants 12-01-00704, 12-01-31342 a,
and by the Simons-IUM Grant.
The first author was supported by the Dynasty Foundation.
The third author was supported by a fellowship of the Alexander von Humboldt Foundation.
}
\address{Department of Higher Algebra, Faculty of Mechanics and Mathematics,
Lomonosov Moscow State University, Leninskie Gory 1, GSP-1, Moscow, 119991,
Russia}
\address{
National Research University Higher School of Economics,
School of Applied Mathematics and Information Science,
Bolshoi Trekhsvyatitelskiy~3, Moscow 109028, Russia}

\email{arjantse@mccme.ru}

\author{Alexander Perepechko}
%\thanks{}
\address{Department of Higher Algebra, Faculty of Mechanics and Mathematics,
Lomonosov Moscow State University, Leninskie Gory 1, GSP-1, Moscow, 119991,
Russia}
\address{Universit\'e Grenoble I, Institut Fourier, UMR 5582 CNRS-UJF, BP 74, 34802 St Martin d'H\`eres c\'edex, France}  \email{perepeal@gmail.com}
\author{Hendrik S\"u\ss}
\address{School of Mathematics,
The University of Edinburgh,
James Clerk Maxwell Building,
The King's Buildings,
Mayfield Road,
Edinburgh EH9 3JZ } \email{suess@sdf-eu.org}
\date{\today}
\begin{abstract}
Let $X$ be an algebraic variety covered by open charts isomorphic to the affine space
and $q: \widehat{X} \to X$ be the universal torsor over $X$. We prove that
the special automorphism group of the quasi-affine variety $\widehat{X}$ acts on $\widehat{X}$
infinitely transitively. Also we find wide classes of varieties $X$ admitting such a covering.
\end{abstract}
\subjclass[2010]{Primary 14M25, 14R20; \ Secondary 14J50, 14L30}
\keywords{Algebraic variety, universal torsor, Cox ring, automorphism, transitivity}
\maketitle
\section*{Introduction}

The aim of this paper is to investigate the infinite transitivity property for
the special automorphism group of universal torsors over smooth algebraic varieties.

Recall that an action of a group $G$ on a set $Y$
is $m$-transitive if it is transitive on $m$-tuples of pairwise distinct points in $Y$,
and is infinitely transitive if it is $m$-transitive for all positive integers $m$.
If $Y$ is an algebraic variety, let us denote by $\SAut(Y)$ the group of special automorphisms
of $Y$, i.e. the subgroup of the automorphism group $\Aut(Y)$ generated by all one-parameter
unipotent subgroups. Assume that $Y$ is an affine irreducible variety of dimension~$\ge 2$.
It is proved in~\cite{AFKKZ} that transitivity of the action of
$\SAut(Y)$ on the smooth locus $Y_{\reg}$ implies infinite transitivity of this action.
Moreover, these conditions are equivalent to flexibility of~$Y$, which is a local
property formulated in terms of velocity vectors to orbits of one-parameter unipotent subgroups.

The study of flexible varieties is important for several reasons. The infinite transitivity
property shows that the group of (special) automorphisms is "large" in this case.
It may indicate that these varieties are the most interesting ones from the geometric point
of view. In the recent paper by Bogomolov, Karzhemanov and Kuyumzhiyan~\cite{BKK} the connection
between flexibility and unirationality is investigated. It is proved in~\cite{AFKKZ}
that every flexible variety is unirational. As a result in the opposite direction, it is conjectured in~\cite{BKK} that any unirational variety is stably birational to some
infinitely transitive variety. This conjecture is confirmed in~\cite{BKK} for several
important cases.

It is easy to show that the affine space $\AA^n$ is flexible. Kaliman and Zaidenberg~\cite{KZ} proved that any hypersurface in $\AA^{n+2}$ given by equation $uv=f(x_1,\ldots,x_n)$ with a
non-constant polynomial $f$ has the infinite transitivity property. More generally, if $X$ is a flexible affine variety, then the suspension over $X$, i.e. a hypersurface in $\AA^2\times X$
given by equation $uv=f(x)$, is flexible as well~\cite{AKZ}. Also it is shown in~\cite{AKZ} that
any non-degenerate affine toric variety is flexible. By~\cite{AFKKZ}, flexibility holds for
affine homogeneous spaces of semisimple algebraic groups and total spaces of vector bundles over
flexible varieties. In~\cite{Pe}, flexibility is established for affine cones
over some del Pezzo surfaces. Some other examples of flexible varieties can be found
in~\cite{AFKKZ}, \cite{AFKKZ-2}.

In general, one would like to have geometric constructions of flexible varieties.
In this paper we prove that the universal torsor over a variety admitting a
covering by affine spaces leads to a flexible variety and thus obtain
a wide class of quasi-affine varieties with infinite transitivity property.

Let $X$ be a smooth algebraic variety. Assume that the divisor class group $\Cl(X)$ is
a lattice of rank $r$. The universal torsor $q: \widehat{X} \to X$ is a locally trivial
$H$-principal bundle with certain characteristic properties, where $H$ is an algebraic
torus of dimension $r$, see \cite[Section~1]{Sk};
here $\widehat{X}$ is a smooth quasi-affine algebraic variety.

Universal torsors were introduced by Colliot-Th\'el\`ene and Sansuc
in the framework of arithmetic geometry to investigate rational points on
algebraic varieties, see~\cite{CTS1}, \cite{CTS2}, \cite{Sk}. In the last years
they were used to obtain positive results on Manin's Conjecture. Another source of interest
is Cox's paper~\cite{Cox}, where an explicit description of the universal torsor
over a toric variety is given. This approach had an essential impact on toric geometry.
For generalizations and relations to Cox rings, see \cite{HK}, \cite{BH1}, \cite{BH2}, \cite{Ha}, \cite{ADHL}.

%The aim of this paper is to show that under some mild restrictions on $X$ the
%automorphism group $\Aut(\widehat{X})$ acts on $\widehat{X}$ infinitely transitively.
%We use a construction of~\cite{KPZ} to show that open cylindric subsets
%on $X$ define one-parameter unipotent subgroups in $\Aut(\widehat{X})$. It turns out that
%if $X$ is transversally covered by cylinders, then the group $\SAut(\widehat{X})$ acts on $\widehat{X}$ transitively.
%
%The next task
%is to prove that transitivity implies infinite transitivity. To this end, we generalize some results of~\cite{AFKKZ} from affine to quasi-affine case.

The paper is organized as follows. In Section~\ref{sec1} we recall basic definitions and facts
on Cox rings and universal torsors.
The group of special automorphisms $\SAut(Y)$ of an algebraic variety $Y$ is considered
in Section~\ref{sec2}. It is shown in \cite{AFKKZ} that if $Y$ is affine of
dimension at least 2 and the group $\SAut(Y)$ acts transitively on an open subset in $Y$,
then this action is infinitely transitive. In Theorem~\ref{mainq} we extend this result
to the case when $Y$ is quasi-affine.

It is observed in~\cite{KPZ} that open cylindric subsets on a projective variety $X$ give rise
to one-parameter unipotent subgroups in the automorphism group of an affine cone over $X$.
This idea is developed further in \cite{KPZ2} and \cite{Pe}.
In Section~\ref{sec3} we show that if
$X$ is a smooth algebraic variety with a free finitely generated
divisor class group $\Cl(X)$, which is transversally covered by cylinders,
then the group $\SAut(\widehat{X})$ acts on the universal torsor $\widehat{X}$ transitively.

As a particular case, in
Section~\ref{sec4} we study $A$-covered varieties, i.e. varieties covered by open subsets
isomorphic to the affine space. Clearly, any $A$-covered variety is smooth and rational.
We list wide classes of $A$-covered varieties including smooth complete
toric or, more generally, spherical varieties, smooth rational projective surfaces, and
some Fano threefolds. It~is shown that the condition to be $A$-covered is preserved under
passing to vector bundles and their projectivizations as well as to the blow up in
a linear subvariety.
In the appendix to this paper we prove that every smooth complete rational variety with a torus action of complexity one is $A$-covered. This part uses the technique of polyhedral divisors from
\cite{ah06}, \cite{ahs}.

In Section~\ref{sec5} we summarize our results on universal torsors and infinite transitivity.
Theorem~\ref{tmain} claims that if $X$ is an $A$-covered algebraic variety of dimension
at least~2, then $\SAut(\widehat{X})$ acts on the universal torsor $\widehat{X}$
infinitely transitively. If the Cox ring $\mathcal{R}(X)$
is finitely generated, then the total coordinate space
$\overline{X}:=\Spec\,\mathcal{R}(X)$ is a factorial affine variety,
the group $\SAut(\overline{X})$ acts on $\overline{X}$ with an open orbit $O$,
and the action of $\SAut(\overline{X})$ on $O$ is infinitely transitive, see Theorem~\ref{ttmain}.
In particular, the Makar-Limanov invariant of $\overline{X}$ is trivial, see Corollary~\ref{corML}.

We work over an algebraically closed field $\KK$ of characteristic zero.

\section{Preliminaries on Cox rings and universal torsors}
\label{sec1}

Let $X$ be a normal algebraic variety with free finitely generated divisor class group
$\Cl(X)$. Denote by $\WDiv(X)$ the group of Weil divisors on $X$ and fix a subgroup
$K\subseteq \WDiv(X)$ such that the canonical map $c\colon K\to\Cl(X)$ sending $D\in K$
to its class $[D]\in\Cl(X)$ is an isomorphism. We define the {\it Cox sheaf} associated
to $K$ to be
$$
\mathcal{R} := \bigoplus_{[D]\in\Cl(X)} \mathcal{R}_{[D]},
\qquad \mathcal{R}_{[D]} := \Of_X(D),
$$
where $D\in K$ represents $[D]\in\Cl(X)$ and the multiplication in $\mathcal{R}$
is given by multiplying homogeneous sections in the field of rational functions $\KK(X)$.
The sheaf $\mathcal{R}$ is a quasicoherent sheaf of normal integral $K$-graded
$\Of_X$-algebras and, up to isomorphy, it does not depend on the choice of the subgroup
$K\subseteq\WDiv(X)$, see \cite[Construction~I.4.1.1]{ADHL}. The {\it Cox ring} of $X$
is the algebra of global sections
$$
\mathcal{R}(X) := \bigoplus_{[D]\in\Cl(X)} \mathcal{R}_{[D]}(X),
\qquad \mathcal{R}_{[D]}(X) := \Gamma(X,\Of_X(D)).
$$

Let us assume that $X$ is a smooth variety with only constant invertible functions.
Then the sheaf $\mathcal{R}$ is locally of finite type,
and the relative spectrum $\Spec_X \mathcal{R}$ is a quasi-affine variety $\widehat{X}$,
see \cite[Corollary~I.3.4.6]{ADHL}. We have $\Gamma(\widehat{X},\Of)\cong\mathcal{R}(X)$, and
the ring $\mathcal{R}(X)$ is a unique factorization domain with only constant invertible
elements, see \cite[Proposition~I.4.1.5]{ADHL}.
Since the sheaf $\mathcal{R}$ is $K$-graded,
the variety $\widehat{X}$ carries a natural action of the torus $H:=\Spec\,\KK[K]$.
The projection $q\colon\widehat{X}\to X$ is called the {\it universal torsor}
over the variety $X$. By \cite[Remark~I.3.2.7]{ADHL}, the morphism
$q\colon\widehat{X}\to X$ is a locally trivial $H$-principal bundle.
In particular, the torus $H$ acts on $\widehat{X}$ freely.

\begin{lemma} \label{leman}
Let $X$ be a normal variety. Assume that there is an open subset
$U$ on $X$ which is isomorphic to the affine space $\AA^n$. Then
any invertible function on $X$ is constant and
the group $\Cl(X)$ is freely generated by classes
$[D_1],\ldots,[D_k]$ of the prime divisors such that
$$
X\setminus U = D_1 \cup \ldots \cup D_k.
$$
\end{lemma}

\begin{proof}
The restriction of an invertible function to $U$ is constant, so the function
is constant.
Since $U$ is factorial, any Weil divisor on $X$ is linearly equivalent
to a divisor whose support does not intersect $U$. This shows that
the group $\Cl(X)$ is generated by $[D_1],\ldots,[D_k]$.

Assume that
$a_1D_1+\ldots+a_kD_k=\div(f)$ for some $f\in\KK(X)$. Then $f$ is a regular
invertible function on $U$ and thus $f$ is a constant. This shows that
the classes $[D_1],\ldots,[D_k]$ generate the group $\Cl(X)$ freely.
\end{proof}

The Cox ring $\mathcal{R}(X)$ and the relative spectrum $q\colon \widehat{X}\to X$
can be defined and studied under weaker assumptions on the variety $X$,
see \cite[Chapter~I]{ADHL}. But in this paper we are interested in smooth varieties
with free finitely generated divisor class group.

Assume that the Cox ring $\mathcal{R}(X)$ is finitely generated. Then we may consider
the {\it total coordinate space} $\overline{X}:=\Spec\,\mathcal{R}(X)$.
This is a factorial affine $H$-variety. By~\cite[Construction~I.6.3.1]{ADHL},
there is a natural open $H$-equivariant embedding $\widehat{X}\hookrightarrow\overline{X}$
such that the complement $\overline{X}\setminus\widehat{X}$ is of codimension
at least two.

%%%%%%%%%%%%%%%%%%%%%%%%%%%%%%%%%%%%%%%%%%%%%%%%%%%%%%%%%%%%%%%%

\section{Special automorphisms and infinite transitivity}
\label{sec2}

An action of a group $G$ on a set $A$ is said to be {\it $m$-transitive} if
for every two tuples of pairwise distinct points $(a_1,\ldots,a_m)$
and $(a_1',\ldots,a_m')$ in $A$ there exists $g\in G$ such that
$g\cdot a_i=a_i'$ for $i=1,\ldots,m$. An action which is $m$-transitive
for all $m\in\ZZ_{>0}$ is called {\it infinitely transitive}.

Let $Y$ be an algebraic variety.
Consider a regular action $\GG_a\times Y \to Y$ of the
additive group $\GG_a=(\KK,+)$ of the ground field on~$Y$. The image $L$
of $\GG_a$ in the automorphism group $\Aut(Y)$ is a one-parameter
unipotent subgroup. We let $\SAut(Y)$ denote the subgroup of~$\Aut(Y)$
generated by all its one-parameter unipotent subgroups. Automorphisms
from the group $\SAut(Y)$ are called {\it special}. In general,
$\SAut(Y)$ is a normal subgroup of~$\Aut(Y)$.

Denote by $Y_{\reg}$ the smooth locus of a variety $Y$.
We say that a point
$y\in Y_{\reg}$ is {\it flexible} if the tangent space $T_yY$ is
spanned by the tangent
vectors to the orbits $L\cdot y$ over all one-parameter unipotent
subgroups $L$ in $\Aut(Y)$.
The variety $Y$ is {\it flexible} if every point $y\in Y_{\reg}$ is.
Clearly, $Y$ is flexible if one point of $Y_{\reg}$ is and the group $\Aut(Y)$
acts transitively on $Y_{\reg}$.

The following result is proved in~\cite[Theorem~0.1]{AFKKZ}.

\begin{theorem}\label{main}
Let $Y$ be an irreducible affine variety of dimension $\ge 2$.
Then the following
conditions are equivalent.
\begin{enumerate}
\item
The group $\SAut(Y)$ acts transitively on $Y_{\reg}$.
\item
The group $\SAut(Y)$ acts infinitely transitively on $Y_{\reg}$.
\item
The variety $Y$ is flexible.
\end{enumerate}
\end{theorem}

A more general version of implication 1~$\Rightarrow$~2 is given
in~\cite[Theorem~2.2]{AFKKZ}. In this section we obtain an analog
of this result for quasi-affine varieties, see Theorem~\ref{mainq} below.

Let $Y$ be an algebraic variety. A regular action $\GG_a\times Y\to Y$
defines a structure of a rational $\GG_a$-algebra on $\Gamma(Y,\Of)$. The differential
of this action is a locally nilpotent derivation $D$ on $\Gamma(Y,\Of)$.
Elements in $\Ker D$ are precisely the functions invariant under $\GG_a$.
The structure of a $\GG_a$-module on $\Gamma(Y,\Of)$ can be reconstructed
from $D$ via exponential map.

Assume that $Y$ is quasi-affine. Then regular functions separate points on $Y$.
In particular, any automorphism of $Y$ is uniquely defined by the induced
automorphism of the algebra $\Gamma(Y,\Of)$. Hence a regular $\GG_a$-action
on $Y$ can be reconstructed from the corresponding locally nilpotent derivation $D$.
At the same time, if $Y$ is not affine, then not every locally nilpotent derivation
on $\Gamma(Y,\Of)$ gives rise to a regular $\GG_a$-action on $Y$. For example, the derivation
$\frac{\partial}{\partial x_1}$ does not define a regular $\GG_a$-action on
$\AA^2\setminus\{(0,0)\}$, while $x_2\frac{\partial}{\partial x_1}$ does.

If $D$ is a locally nilpotent derivation assigned to a $\GG_a$-action on
a quasi-affine variety $Y$ and $f\in\Ker D$, then the derivation $fD$ is locally nilpotent and
it corresponds to a $\GG_a$-action on $Y$ with the same orbits on $Y\setminus\div(f)$,
which fixes all points on the divisor $\div(f)$. The one-parameter subgroup of $\SAut(Y)$
defined by $fD$ is called a {\em replica} of the subgroup given by $D$.

We say that a subgroup $G$ of $\Aut(Y)$ is {\em algebraically generated} if it is generated
as an abstract group by a family $\cG$ of connected algebraic\footnote{not necessarily affine.} subgroups of $\Aut(Y)$.

\begin{proposition} \cite[Proposition~1.5]{AFKKZ} \label{1.2}
There are (not necessarily distinct) subgroups
$H_1,\ldots, H_s\in \cG$ such that
\begin{equation} \label{1.2.a}
G.x= (H_1\cdot H_2\cdot\ldots\cdot H_s)\cdot x\quad \forall x\in X.
\end{equation}
\end{proposition}

A sequence $\cH=(H_1,\ldots , H_s)$
satisfying condition \eqref{1.2.a} of Proposition~\ref{1.2} is called {\em complete}.

Let us say that a subgroup $G\subseteq\SAut(Y)$ is {\em saturated} if it is generated
by one-parameter unipotent subgroups and there is a complete sequence $(H_1,\ldots,H_s)$
of one-parameter unipotent subgroups in $G$ such that $G$ contains all replicas of
$H_1,\ldots,H_s$. In particular, $G=\SAut(X)$ is a saturated subgroup.

\begin{theorem} \label{mainq}
Let $Y$ be an irreducible quasi-affine
algebraic variety of dimension $\ge 2$ and let  $G\subseteq\SAut(Y)$ be a saturated subgroup,
which acts with an open orbit $O\subseteq Y$. Then $G$ acts on~$O$ infinitely transitively.
\end{theorem}

\begin{remark} \label{exanew}
Let $H$ be a one-parameter unipotent subgroup of $G$.
According to \cite[Theorem~3.3]{PV}, the field of rational invariants $\KK(Y)^H$ is the field
of fractions of the algebra $\KK[Y]^H$ of regular invariants. Hence, by Rosenlicht's Theorem (see~\cite[Proposition~3.4]{PV}), regular invariants separate orbits on an $H$-invariant open
dense subset $U(H)$ in $Y$. Furthermore, $U(H)$ can be chosen to be contained in $O$ and
consisting of 1-dimensional $H$-orbits.
\end{remark}

For the remaining part of this section we fix the following notation.
Let $H_1,\ldots, H_s$ be a complete sequence of one-parameter unipotent subgroups in $G$.
We choose subsets $U(H_1),\ldots,U(H_s)\subseteq O$ as in Remark~\ref{exanew} and let
$$
V=\bigcap_{k=1}^s U(H_k)\,.
$$
In particular, $V$ is open and dense in $O$. We say that a set of points $x_1,\ldots,x_m$
in $Y$ is {\em regular}, if $x_1,\ldots,x_m\in V$ and $H_k\cdot x_i\ne H_k\cdot x_j$ for all
$i,j=1,\ldots,m$, $i\ne j$, and all $k=1,\ldots,s$.

\begin{remark} \label{fiberwise}
For any $H_k$, any 1-dimensional $H_k$-orbit $O_1,\ldots, O_r$ intersecting $V$ and any $p=1,\ldots,s$ we may choose a replica $H_{k,p}$ such that all $O_q$ but $O_p$ are pointwise $H_{k,p}$-fixed. To this end,
we find $H_k$-invariant functions $f_{k,p,p'}$ such that $f_{k,p,p'}|_{O_p}=1$, $f_{k,p,p'}|_{O_{p'}}=0$.
Then we take
$$
H_{k,p}=\{\, \exp(t(\prod_{p'\ne p}f_{k,p,p'}) D_k) \ ; \ t\in\KK\, \},
$$
where $D_k$ is a locally nilpotent derivation corresponding to $H_k$.
\end{remark}

\begin{lemma}\label{lemV}
For every points $x_1,\ldots,x_m\in O$ there exists an element $g\in G$ such that
the set $g\cdot x_1,\ldots,g\cdot x_m$ is regular.
\end{lemma}

\begin{proof}
For any $x_i$ there holds $V\subset O=H_1\cdots H_s \cdot x_i$. The condition $h_1\cdots h_s \cdot x\in V$ is open and nonempty, hence we obtain an open subset $W\subset H_1\times\ldots\times H_s$ such that $h_1\cdots h_s \cdot x_i\in V$ for all
$(h_1,\ldots,h_s)\in W$ and all $x_i$.

So we may suppose that $x_1,\ldots,x_m\in V$. Let $N$ be the number of triples $(i,j,k)$
such that $i\ne j$ and $H_k\cdot x_i=H_k\cdot x_j$.
If $N=0$ then the lemma is proved. Assume that $N\ge 1$ and fix such a triple $(i,j,k)$.

There exists $l$ such that $H_k\cdot x_i$ has at most finite intersection with $H_l$-orbits;
otherwise $H_k\cdot x_i$ is invariant with respect to all $H_1,\ldots,H_s$,
a contradiction with the condition $\dim O\ge 2$.

We claim that there is a one-parameter subgroup $H$ in $G$ such that
\begin{equation} \label{Hk-condition}
H_k\cdot (h\cdot x_i)\neq H_k\cdot (h\cdot x_j)\quad \mbox{ for all
but finitely many elements} \ h\in H.
\end{equation}

Let us take first $H=H_l$.
Condition (\ref{Hk-condition}) is determined by a finite set of $H_k$-invariant functions. So, either it holds or $H_k\cdot (h\cdot x_i)= H_k\cdot (h\cdot x_j)$ for all $h\in H$.

Assume that $H_l\cdot x_i\neq H_l \cdot x_j$. By Remark~\ref{fiberwise} there exists a replica $H_l^\prime$ such that $H_l^\prime\cdot x_i=x_i$, but $H_l^\prime\cdot x_j=H_l\cdot x_j$. We take $H=H_l^\prime$, and condition (\ref{Hk-condition}) is fulfilled.

Assume now the contrary. Then there exists $h_l\in H_l$ such that $h_l\cdot x_i=x_j$.
Then the set $\{h_l^n\cdot x_i\;|\; n\in\ZZ_{>0}\}$ has finite intersection with any $H_k$-orbit, and $h_l^n\cdot x_j=h_l^{n+1}\cdot x_i$ lie in different $H_k$-orbits for an infinite set of $n\in\ZZ_{>0}$. Therefore, this holds for an open subset of $H_l$, and condition (\ref{Hk-condition}) is again fulfilled.

Finally, the following conditions are open and nonempty on $H$:
\begin{enumerate}
%\item $h_l\cdot x_\beta\notin H_k\cdot x_\beta$.
\item[(C1)] $h\cdot x_1,\ldots,h\cdot x_m\in V$;
\item[(C2)] if $H_p\cdot x_{i'}\neq H_p\cdot x_{j'}$ for some $p$ and $i'\neq j'$,
then $H_p\cdot (h\cdot x_{i'})\neq H_p\cdot (h\cdot x_{j'})$.
\end{enumerate}

Hence there exists $h\in H$ satisfying (C1), (C2), and condition
(\ref{Hk-condition}). We conclude that for the set $(h\cdot x_1,\ldots,h\cdot x_m)$
the value of $N$ is smaller, and proceed by induction.
\end{proof}

\begin{lemma} \label{lemz}
Let $x_1,\ldots,x_m$ be a regular set and $G(x_1,\ldots,x_{m-1})$ be the intersection
of the stabilizers of the points $x_1,\ldots,x_{m-1}$ in $G$. Then the orbit
$G(x_1,\ldots,x_{m-1})\cdot x_m$ contains an open subset in $O$.
\end{lemma}

\begin{proof}
We claim that there is a nonempty open subset $U\subseteq H_1\times\ldots\times H_s$ such
that for every $(h_1,\ldots,h_s)\in U$ we have
$$
h_1\ldots h_s\cdot x_m =g\cdot x_m \quad \text{for some} \ g\in G(x_1,\ldots,x_{m-1}).
$$
Indeed, let $Z$ be the union of orbits $H_k\cdot x_i$, $k=1,\ldots,s$, $i=1,\ldots,m-1$.
The set $V\setminus Z$ is open and contains $x_m$. Let $U$ be the set of all
$(h_1,\ldots,h_s)$ such that $h_r\ldots h_s\cdot x_m\in V\setminus Z$ for any
$r=1,\ldots,s$. Then $U$ is open and nonempty. Let us show that for any
$(h_1,\ldots,h_s)\in U$ and any $r=1,\ldots,s$ the point $h_r\ldots h_s\cdot x_m$
is in the orbit $G(x_1,\ldots,x_{m-1})\cdot x_m$. Assume that
$h_{r+1}\ldots h_s\cdot x_m\in G(x_1,\ldots,x_{m-1})\cdot x_m$.
By Remark~\ref{fiberwise}, there is a replica $H_r'$ of the subgroup
$H_r$ which fixes $x_1,\ldots,x_{m-1}$ and such that the orbits
$$
H_r\cdot (h_{r+1}\ldots h_s\cdot x_m) \quad \text{and} \quad
H_r'\cdot (h_{r+1}\ldots h_s\cdot x_m)
$$
coincide. Then $H_r'$ is contained in $G(x_1,\ldots,x_{m-1})$ and the point
$h_rh_{r+1}\ldots h_s\cdot x_m$ is in the orbit $G(x_1,\ldots,x_{m-1})\cdot x_m$
for any $h_r\in H_r$.
The claim is proved.

Now the image of the dominant morphism
$$
U\to O, \quad (h_1,\ldots,h_s) \mapsto h_1\ldots h_s\cdot x_m
$$
contains an open subset in $O$.
\end{proof}

\begin{proof}[Proof of Theorem~\ref{mainq}]
Let $(x_1,\ldots,x_m)$ and $(y_1,\ldots,y_m)$ be two sets of pairwise
distinct points in $O$. We have to show that there is an element $g\in G$ such that
$g\cdot x_1=y_1,\ldots,g\cdot x_m=y_m$.

We argue by induction on $m$. If $m=1$, then the claim is obvious. If $m>1$, then
by inductive hypothesis there exists $g'\in G$ such that $g'\cdot x_1=y_1,\ldots,
g'\cdot x_{m-1}=y_{m-1}$. If $g'\cdot x_m=y_m$, the assertion is proved. Assume that
$g'\cdot x_m\ne y_m$. By Lemma~\ref{lemV}, there exists $g''\in G$ such that
the set
$$
g''\cdot y_1,\,\ldots,g''\cdot y_{m-1},\,g''\cdot y_m,\,g''g'\cdot x_m
$$
is regular. Lemma~\ref{lemz} implies that the orbits
$$
G(g''\cdot y_1,\ldots,g''\cdot y_{m-1})\cdot (g''\cdot y_m) \quad \text{and}
\quad G(g''\cdot y_1,\ldots,g''\cdot y_{m-1})\cdot (g''g'\cdot x_m)
$$
intersect, so there is $g'''\in G(g''\cdot y_1,\ldots,g''\cdot y_{m-1})$ such that
$g'''g''g'x_m=g''y_m$. Then the element $g=(g'')^{-1}g'''g''g'$ is as desired.
\end{proof}

%%%%%%%%%%%%%%%%%%%%%%%%%%%%%%%%%%%%%%%%%%%%%%%%%%%%%%%%%%%%%%%

\section{Cylinders and $\GG_a$-actions}
\label{sec3}

The following definition is taken from~\cite{KPZ}, see also~\cite{KPZ2}.

\begin{definition}
Let $X$ be an algebraic variety and $U$ be an open subset of $X$.
We say that $U$ is a {\it cylinder} if $U\cong Z\times\AA^1$,
where $Z$ is an irreducible affine variety with $\Cl(Z)=0$.
\end{definition}

\begin{proposition} \label{cylact}
Let $X$ be a smooth algebraic variety with a free finitely generated
divisor class group $\Cl(X)$, $q\colon\widehat{X}\to X$ be
the universal torsor, and $U\cong Z\times\AA^1$ be a cylinder in $X$. Then
there is an action $\GG_a\times\widehat{X}\to\widehat{X}$ such that
\begin{enumerate}
\item[(i)]
the set of $\GG_a$-fixed points is $\widehat{X}\setminus q^{-1}(U)$;
\item[(ii)]
for any point $y\in q^{-1}(U)$ the isomorphism $U\cong Z\times\AA^1$
identifies the subset $q(\GG_a\cdot y)$
with a fiber of the projection $Z\times\AA^1\to Z$.
%if $L$ is the image of $\GG_a$ in $\Aut(\widehat{X})$, then
%for any point $y\in q^{-1}(U)$ the isomorphism $U\cong Z\times\AA^1$
%identifies the subset $q(L\cdot y)$
%with the fiber $\{z\}\times\AA^1$ of the projection $Z\times\AA^1\to Z$
%over some point $z\in Z$.
\end{enumerate}
\end{proposition}

\begin{proof}
Since $\Cl(U)\cong\Cl(Z)=0$, we have an isomorphism $q^{-1}(U)\cong Z\times\AA^1\times H$
compatible with the projection $q$,
see \cite[Remark~I.3.2.7]{ADHL}. Thus the subset $q^{-1}(U)$ admits a $\GG_a$-action
$$
a\cdot (z,t,h)=(z,t+a,h), \quad z\in Z, \ t\in\AA^1, \ h\in H,
$$
with property~(ii). Denote by $D$ the locally nilpotent derivation on
$\Gamma(U,\Of)$ corresponding to this action.

Our aim is to extend the action to $\widehat{X}$. Since the open subset
$q^{-1}(U)$ is affine, its complement $\widehat{X}\setminus q^{-1}(U)$
is a divisor $\Delta$ in $\widehat{X}$. We can find a function
$f\in\Gamma(\widehat{X},\Of)$ such that $\Delta=\div(f)$. In particular,
$$
\Gamma(q^{-1}(U),\Of)=\Gamma(\widehat{X},\Of)[1/f].
$$
Since $f$ has no zero on any $\GG_a$-orbit on $q^{-1}(U)$, it is constant along orbits,
and $f$ lies in $\Ker D$.

\begin{lemma}\label{imp}
Let $Y$ be an irreducible quasi-affine variety,
$$
Y=\bigcup_{i=1}^s\ Y_{g_i}, \qquad g_i\in\Gamma(Y,\Of),
$$
be an open covering by principle affine subsets, and let
$$
\Gamma(Y_{g_i},\Of)=\KK[c_{i1},\ldots,c_{ir_i}][1/g_i]
$$
for some $c_{ij}\in\Gamma(Y,\Of)$. Consider a finitely generated subalgebra $C$ in
$\Gamma(Y,\Of)$ containing all the functions $g_i$ and $c_{ij}$. Then
the natural morphism $Y\to\Spec\,C$ is an open embedding.
\end{lemma}

\begin{proof}
Notice that $\Gamma(Y_{g_i},\Of)=\Gamma(Y,\Of)[1/g_i]=C[1/g_i]$.
This shows that the morphism \linebreak $Y\to\Spec\,C$ induces isomorphisms
$Y_{g_i}\cong (\Spec\,C)_{g_i}$.
\end{proof}

Let $Y=\widehat{X}$ and $\widehat{X}\hookrightarrow \Spec\,C$ be an
affine embedding as in Lemma~\ref{imp} with $f\in C$.
A finite generating set of the algebra $C$ is contained in a
finite dimensional $D$-invariant subspace $W$ of $\Gamma(q^{-1}(U),\Of)$.
Replacing $D$ with $f^mD$ we may assume that $W$ is contained in $\Gamma(\widehat{X},\Of)$.
We enlarge $C$ and assume that it is generated by $W$. Then $C$ is an $(f^mD)$-invariant
finitely generated subalgebra in $\Gamma(\widehat{X},\Of)$ and we have an open
embedding $\widehat{X}\hookrightarrow \Spec\,C=:\widetilde{X}$.

Replacing $f^mD$ with $D':=f^{m+1}D$, we obtain a locally nilpotent derivation $D'$ on $C$
such that $D'(C)$ is contained in $fC$.
The~corresponding $\GG_a$-action on $\widetilde{X}$ fixes all points on $\div(f)$
and has the same orbits on $q^{-1}(U)$. Hence the subset $\widehat{X}\subseteq\widetilde{X}$
is $\GG_a$-invariant and the restriction of the action to $\widehat{X}$
has the desired properties.
The proof of Proposition~\ref{cylact} is completed.
\end{proof}

\begin{remark}
Under the assumption that the algebra $\Gamma(\widetilde{X},\Of)$ is finitely
generated the proof of Proposition~\ref{cylact} is much simpler.
\end{remark}

The following definitions appeared in \cite{Pe}.

\begin{definition} \label{d1}
Let $X$ be a variety and $U\cong Z\times\AA^1$ be a cylinder in $X$.
A subset $W$ of $X$ is said to be {\it $U$-invariant} if
$W\cap U=p_1^{-1}(p_1(W\cap U))$, where $p_1\colon U\to Z$ is the projection
to the first factor. In other words, every $\AA^1$-fiber of the cylinder
is either contained in $W$ or does not meet $W$.
\end{definition}

\begin{definition}
We say that a variety $X$ is {\it transversally covered} by cylinders $U_i$,
$i=1,\ldots,s$, if $X=\bigcup_{i=1}^s U_i$ and there is no proper subset $W\subset X$
invariant under all $U_i$.
\end{definition}

%In the next section we consider an important particular class
%of varieties transversally covered by cylinders.
\begin{proposition} \label{propostouse}
Let $X$ be a smooth algebraic variety with a free finitely generated
divisor class group $\Cl(X)$ and $q\colon\widehat{X}\to X$ be
the universal torsor. Assume that $X$ is transversally covered by cylinders.
Then the group $\SAut(\widehat{X})$ acts on $\widehat{X}$ transitively.
\end{proposition}

\begin{proof}
Consider a $\GG_a$-action on $\widehat{X}$ associated with
the cylinder $U_i$ as in Proposition~\ref{cylact}.
Let $L_i$ be the corresponding $\GG_a$-subgroup in $\SAut(\widehat{X})$
and $G$ be the subgroup of $\SAut(\widehat{X})$ generated by all the $L_i$.

By Proposition~\ref{cylact}, the projection
of any $G$-orbit on $\widehat{X}$ to $X$ is invariant under all the~cylinders~$U_i$,
and thus this projection coincides with $X$.
In particular, every $\SAut(\widehat{X})$-orbit $S$ on $\widehat{X}$
projects to $X$ surjectively.

Let $H_S$ be the stabilizer of the subset $S$ in~$H$.
The subgroup $\SAut(\widehat{X})$ is normalized by $H$. This yields that
if for some $x\in \widehat{X}$ and $h\in H$ the point $h\cdot x$ lies in $S$,
then $h$ is contained in~$H_S$. In other words, the orbit $S$ intersects every
fibre of the torsor $q\colon \widehat{X}\to X$ in an $H_S$-orbit.

By \cite[Proposition~1.3]{AFKKZ}, any $\SAut(\widehat{X})$-orbit is locally closed in $\widehat{X}$. Since the torus $H$ permutes $G$-orbits, all of them are closed in $\widehat{X}$. This yields that $H_S$ is a closed subgroup of~$H$.

Assume that $H_S$ is a proper subgroup of $H$. Then there is a
nonzero character $\chi\in\XX(H)$ such that $\chi|_{H_S}=1$.
Consider a trivialization covering $X=U_1\cup\ldots\cup U_r$ of
the bundle $q\colon \widehat{X}\to X$, that is $q^{-1}(U_i)\cong U_i\times H$.
Let
$$
\psi_{ij}\colon U_i\cap U_j \to H
$$
be the transition functions of this bundle. We define a locally trivial
$\KK^{\times}$-bundle $X_{\chi}$
over $X$ by gluing the covering $\{U_i\times\KK^{\times}\}$
with the transition functions $U_i\cap U_j \to \KK^{\times}$, $x\mapsto\chi(\psi_{ij}(x))$.
Then the maps
$$
U_i \times H \to U_i \times \KK^{\times}, \quad (x,h)\mapsto (x, \chi(h))
$$
define a surjective morphism $\widehat{X}\to X_{\chi}$. The image of $S$ under
this morphism intersects every fibre of $X_{\chi}\to X$ in one point.
This shows that the $\KK^{\times}$-bundle $X_{\chi}$ is trivial.
Then the pull back of the coordinate function along a fiber of $X_{\chi}$
is an invertible function on $\widehat{X}$.
Since $\widehat{X}$ has only constant invertible functions, we conclude that
$H_S=H$ and thus $S=\widehat{X}$.
This shows that $\SAut(\widehat{X})$ acts on $\widehat{X}$ transitively.
\end{proof}

%%%%%%%%%%%%%%%%%%%%%%%%%%%%%%%%%%%%%%%%%%%%%%%%%%%%%%%%%%%%%%%

\section{$A$-covered varieties}
\label{sec4}

The affine space $\AA^n$ admits $n$ coordinate cylinder structures
$\AA^{n-1}\times\AA^1$, and the covering of $\AA^n$ by these cylinders
is transversal. This elementary observation motivates the following definition.

\begin{definition}
An irreducible algebraic variety $X$ is said to be {\it $A$-covered}
if there is an open covering $X=U_1\cup\ldots\cup U_r$, where every
chart $U_i$ is isomorphic to the affine space $\AA^n$.
\end{definition}

A choice of such a covering together with isomorphisms $U_i \cong \AA^n$ is called an \emph{$A$-atlas} of~$X$.
A subvariety $Z$ of an $A$-covered variety $X$ is called \emph{linear} with respect to an $A$-atlas, if it is linear in all charts, i.e. $Z \cap U_i$  is a linear subspace in
$U_i \cong \AA^n$.
Any $A$-covered variety is rational, smooth, and by Lemma~\ref{leman} the group
$\Pic(X)=\Cl(X)$ is finitely generated and free.

\smallskip

Clearly, the projective space $\PP^n$ is $A$-covered.
This fact can be generalized in several ways.

\smallskip

\begin{enumerate} [1)]
\setcounter{enumi}{0}
\item \label{item:toric} Every smooth complete toric variety $X$ is $A$-covered.
\smallskip
\item \label{item:tvars} Every smooth rational complete variety with a torus action
of complexity one is $A$-covered; see the appendix to this paper.
\smallskip
\item \label{item:flag} Let $G$ be a semisimple algebraic group and
be $P$ a parabolic subgroup of $G$. Then the flag variety $G/P$
is $A$-covered. Indeed, a maximal unipotent subgroup $N$ of $G$ acts
on $G/P$ with an open orbit $U$ isomorphic to an affine space. Since
$G$ acts on $G/P$ transitively, we obtain the desired covering.
\smallskip
\item More generally, every smooth complete spherical variety is $A$-covered,
see \cite[Corollary~1.5]{BLV}.
%\smallskip
%\item \label{item:prime} The Fano threefolds $\PP^3$, $Q$, $V_5$  and the Mukai-Umemura threefold $V_{22}'$ are known to be $A$-covered.
\smallskip
\item \label{item:product} The product of two $A$-covered varieties is again
$A$-covered.
\smallskip
\item \label{item:bundle} Every vector bundle over $\AA^n$
trivializes, and total spaces of vector bundles over $A$-covered varieties are $A$-covered. The same holds for their projectivizations.
\smallskip
\item \label{item:blowupp} If a variety $X$ is $A$-covered and $X'$ is a blow up of $X$
at some point $p\in X$, then $X'$ is $A$-covered.
\smallskip
\item All smooth projective rational surfaces are obtained
either from $\PP^2$, $\PP^1 \times \PP^1$ or from the Hirzebruch surfaces $F_n$
by a sequence of blow ups of points, and thus they are $A$-covered by \ref{item:blowupp}).
\smallskip
\item  \label{item:blowup}
We may generalize the blow up example as follows. The blow up of $X$ in a linear subvariety $Z$ is $A$-covered. Moreover, the strict transforms of linear subvarieties, which either contain $Z$ or do not intersect with it, are linear again (with the choice of an appropriate $A$-atlas). Hence, we may iterate this procedure.
\end{enumerate}

\begin{proof}[Proof of statement~\ref{item:blowup})]
We consider one chart $U$ of the covering on $X$. We may assume, that we blow up $\AA^n=U$ in the linear subspace given by $x_1=\ldots = x_k=0$. By definition, the blow up $X'$ is given in the product $\AA^n\times\PP^{k-1}$
by equations $x_iz_j=x_jz_i$, where $1\le i,j\le k$.  If the homogeneous coordinate $z_j$ equals $1$ for some $j=1,\ldots,k$, then $x_i=x_jz_i$, and we are in the open chart $V_j$ with
independent coordinates $x_j$, $x_s$ with $s > k$, and $z_i$, $i\ne j$. So the variety $X'$ is covered by $k$ such charts.

Let $L$ be a linear subspace in $U$ containing $[x_1=\ldots = x_k=0]$ and given by linear equations $f_i(x_{1},\ldots,x_k)=0$. The strict transform of $L$ is given in $V_j$ by the equations $f_i(z_{1},\ldots,z_{j-1},1,z_{j+1},\ldots,z_k)=0$. After a change of variables $x_j \mapsto x_j-1$ these equations become linear.

Finally, if a linear subvariety $Z'$ does
not meet the linear subvariety $Z$, then $Z'$ does not intersect charts of our atlas that
intersect $Z$, and the assertion follows.
\end{proof}

\begin{example} \label{ex1}
Consider the quadric threefold $Q$. Choose two points and a conic passing through them. Then these are linear subvarieties of $Q$ with respect to an appropriate atlas. Hence, the iterated blow up in the points, first, and then in the strict transform of the conic is $A$-covered.
This variety has number~4.4 in the classification Fano threefolds, 
see Proposition~\ref{fclass} below.
\end{example}

We may use the above observations to take a closer look at Fano threefolds.

\begin{proposition}\label{fclass}
In the classification of Iskovskikh \cite{isk78} and Mori-Mukai \cite{mm81} we have the following (possibly non-complete) list of $A$-covered Fano threefolds:
\begin{enumerate} [a)]
\item \label{item:primelist} $\PP^3$, $Q$, $V_5$ and the Mukai-Umemura threefold $V_{22}'$;
\item \label{item:toriclist} 2.33-2.36, 3.26-3.31, 4.9-4.11, 5.2, 5.3;
\item \label{item:tvarslist} 2.29, 2.30, 2.31, 2.32, 3.18-3.23, 3.24, 4.4, 4.7, 4.8 and (at least) one element of the family 2.24, 3.8 and 3.10 respectively;
\item \label{item:productlist} 5.3-5.8;
\item \label{item:blowuplist}  2.26.
\end{enumerate}
\end{proposition}

\begin{proof}
Existence of an $A$-covering for varieties in a) can be seen directly from defining equations.
List~\ref{item:toriclist}) are exactly the toric Fano threefolds.

The varieties in~\ref{item:tvarslist}) are precisely non-toric Fano threefolds admitting a
complexity one torus action; this is more or less straightforward to check via the description
in~\cite{mm81} , see also \cite{t3folds}. Thus the claim follows from Theorem~\ref{sec:thm-cplx-one}. In families~2.24 and~3.8 we find a 2-torus action on the hypersurface $V(x_1y_1^2+x_2y_2^2+x_3y_3^2) \subset \PP^2 \times \PP^2$ and on the blow up of
this hypersurface in the curve $(*:1:0,\;0:0:1)$, respectively. Moreover, we have a 2-torus action on the blow up of the quadric $Q=V(x_1x_2+x_3x_4+x_5^2)$ in the conics $C_1=Q \cap [x_1=x_2=0]$
and $C_2=Q \cap [x_3=x_4=0]$. This is an element of family~3.10.

The varieties in~\ref{item:productlist}) are precisely the products of del Pezzo surfaces and $\PP^1$. Finally, the varieties~2.26 are obtained from $V_5$ by blow-up in a linear subvariety as explained in~\ref{item:blowup}).
\end{proof}

\begin{remark}\label{rem:non-a-covered}
For Fano threefolds of Picard rank one the list of $A$-covered ones in
Proposition~\ref{fclass}, a) is almost complete. Indeed, by \cite{furu} the varieties
$\PP^3$, $Q$, $V_5$ and $V_{22}$ are the only possible compactifications of $\AA^3$.
In particular, the Fano threefolds $V_{12}$, $V_{16}$, $V_{18}$ and $V_4$ from Iskovskikh's classification \cite{isk78} are rational but not $A$-covered. The situation remains
unclear only for members of family $V_{22}$ different from the Mukai-Umemura threefold $V_{22}'$.

For higher Picard rank we do not expect the list to be complete, but our arguments 1)-9) do not apply for other than the given examples. Consider, for example, the threefold 4.6. This is a blow-up of $\PP^3$ in three disjoint lines. Here, we cannot apply 9) directly,
since the three lines are not linear with respect to the same A-atlas.
\end{remark}

%%%%%%%%%%%%%%%%%%%%%%%%%%%%%%%%%%%%%%%%%%%%%%%%%%%%%%%%%%%%%%%%

\section{Main results}
\label{sec5}

The following theorem summarizes our results on universal torsors and infinite transitivity.

\begin{theorem} \label{tmain}
Let $X$ be an $A$-covered algebraic variety of dimension at least $2$
and $q\colon\widehat{X}\to X$ be the universal torsor.
Then the group $\SAut(\widehat{X})$ acts on the quasi-affine variety $\widehat{X}$
infinitely transitively.
\end{theorem}

\begin{proof}
If $X$ is covered by $m$ open charts isomorphic
to $\AA^n$, and every chart is equipped with $n$ transversal cylinder structures, then
the covering of $X$ by these $mn$ cylinders is transversal.
By Proposition~\ref{propostouse}, the group $\SAut(\widehat{X})$ acts on
$\widehat{X}$ transitively. Theorem~\ref{mainq} yields that the
action is infinitely transitive.
\end{proof}

Theorem~\ref{tmain} provides many examples of quasi-affine
varieties with rich symmetries.
In particular, if $X$ is a del Pezzo surface, a description of the
universal torsor $q\colon\widehat{X}\to X$ may be found in \cite{BaPo}, \cite{SeSk},
\cite{SeSk2}.
%\cite{SeSk3}, \cite{Zh}.
It follows from Theorem~\ref{tmain}
that the group $\SAut(\widehat{X})$ acts on $\widehat{X}$ infinitely transitively.

If $X$ is the blow up of nine points
in general position on $\PP^2$, that it is well known that the Cox ring
$\mathcal{R}(X)$ is not finitely generated, and thus $\widehat{X}$ is
a quasi-affine variety with a non-finitely generated algebra of
regular functions $\Gamma(\widehat{X},\Of)$. Theorem~\ref{tmain} works in this case as well.

\begin{theorem} \label{ttmain}
Let $X$ be an $A$-covered algebraic variety of dimension at least $2$.
Assume that the Cox ring $\mathcal{R}(X)$
is finitely generated. Then the total coordinate space
$\overline{X}:=\Spec\,\mathcal{R}(X)$ is an affine factorial variety, the group $\SAut(\overline{X})$ acts on $\overline{X}$ with an open orbit $O$,
and the action of $\SAut(\overline{X})$ on $O$ is infinitely transitive.
\end{theorem}

\begin{proof}
Lemma~\ref{leman} shows that the group $\Cl(X)$ is finitely generated and free,
hence the ring $\mathcal{R}(X)$ is a unique factorization domain,
see~\cite[Proposition~I.4.1.5]{ADHL}. Since
$$
\Gamma(\overline{X},\Of)=\mathcal{R}(X)\cong \Gamma(\widehat{X},\Of),
$$
any $\GG_a$-action on $\widehat{X}$ extends to $\overline{X}$.
We conclude that $\widehat{X}$ is contained in one $\SAut(\overline{X})$-orbit $O$
on $\overline{X}$, the action of $\SAut(\overline{X})$ on $O$ is infinitely transitive,
and by \cite[Proposition~1.3]{AFKKZ} the orbit $O$ is open in $\overline{X}$.
\end{proof}

Recall from~\cite{Fr} that the {\it Makar-Limanov invariant} $\text{ML}(Y)$ of an affine
variety $Y$ is the intersection of the kernels of all locally nilpotent derivations
on $\Gamma(Y,\Of)$. In other words $\text{ML}(Y)$ is the subalgebra of
all $\SAut(Y)$-invariants in $\Gamma(Y,\Of)$. Similarly to as in \cite{Lie}
the {\it field Makar-Limanov invariant} $\text{FML}(Y)$ is the subfield of $\KK(Y)$
which consists of all rational $\SAut(Y)$-invariants. If the field Makar-Limanov
invariant is trivial, that is, if $\text{FML}(Y)=\KK$, then so is $\text{ML}(Y)$,
but the converse is not true in general.

\begin{corollary} \label{corML}
Under the assumptions of Theorem~\ref{ttmain} the field Makar-Limanov invariant
$\text{FML}(\overline{X})$ is trivial.
\end{corollary}

\begin{proof}
By Theorem~\ref{ttmain}, the group $\SAut(\overline{X})$ acts on $\overline{X}$
with an open orbit. So any rational $\SAut(\overline{X})$-invariant is constant.
\end{proof}
%%%%%%%%%%%%%%%%%%%%%%%%%%%%%%%%%%%%%%%%%%%%%%%%%%%%%%%%%%%%%%%%%%%%%%

\section*{Appendix: Rational T-varieties of complexity one}
By a $T$-variety we mean a normal variety equipped with an effective action of an algebraic
torus $T$. The difference of dimensions $\dim X -\dim T$ is called the \emph{complexity}
of a $T$-variety. Hence, toric varieties are $T$-varieties of complexity zero.
For the case of complexity one we are going to prove the following theorem.

\begin{theorem} \label{sec:thm-cplx-one}
Any smooth complete rational $T$-variety of complexity one is $A$-covered.
\end{theorem}

Due to \cite{ah06,ahs,tvars} $T$-varieties can be described and studied in the language of polyhedral divisors. Here, we restrict ourself to the case of
rational T-varieties of complexity one. It means that the divisors live on $\PP^1$.
This allows us to simplify some of the definitions.

\subsection*{The affine case}
We consider a lattice $M$ of rank $n$, the dual lattice $N=\text{Hom}(M,\ZZ)$, and
the vector space $N_\QQ = N \otimes_\ZZ \QQ$. Let $T=N \otimes_\ZZ \KK^*$
be the algebraic torus of dimension $n$ with character lattice $M$.

Every polyhedron $\Delta \subset N_\QQ$ has a Minkowski decomposition $\Delta=P + \sigma$,
where $P$ is a (compact) polytope and $\sigma$ is a polyhedral cone.
We call $\sigma$ the \emph{tail cone} of $\Delta$ and denote it by $\tail(\Delta)$. A \emph{polyhedral divisor} on $\PP^1$ over $N$ is a formal sum
\[
\D = \sum_{y \in \PP^1} \D_y \cdot y,
\]
where $\D_y$ are polyhedra with common pointed tail cone $\sigma$ and
only finitely many coefficients differ from $\sigma$ itself.
Note that we allow empty coefficients.

We call $\D$ a \emph{proper} polyhedral divisor or a \emph{p-divisor} for short, if
\begin{equation}
  \label{eq:properness}
  \deg \D := \sum_{y\in \PP^1} \D_y \subsetneq \sigma.
\end{equation}
Here $\deg \D =\emptyset$ if and only if $D_y=\emptyset$ for some $y\in \PP^1$.

By \cite[Theorems~3.1,~3.4]{ah06} there is a functor $X$ associating to a p-divisor $\D$ on $\PP^1$ a rational complexity-one T-variety $X(\D)$ of dimension $n+1$, and every such variety arises this way.

\begin{remark}{\cite[Remark 1.8.]{iv}}
\label{rem:downgrade}
Let us fix two points $y_0,y_\infty \in \PP^1$. For $y \in \PP^1 \setminus \{y_0,y_\infty\}$
we consider lattice points $v_y \in N$ such that only finitely many of them are
different from $0$. We denote the sum $\sum_{y\neq y_0,y_\infty} v_y$ by $v$
and choose $w_0, w_\infty \in N$ with $w_0+w_\infty=v$.

A polyhedral divisor $\D$ of the form
  \begin{equation}
    \label{eq:1}
    \D_0 \cdot y_0 \;+\; \D_\infty \cdot y_\infty \;+\; \sum_y (v_y + \sigma) \cdot y %\tag{*}
  \end{equation}
on $\PP^1$ corresponds to the affine toric variety of the cone
$$\cone(w_0\!+\!\D_0,w_\infty\! +\! \D_\infty) \;:=\;
\QQ_{\geq 0} \cdot \big((w_0\!+\!\D_0) \!\times\! \{1\} \,\cup\, \sigma \!\times\! \{0\} \,
\cup\, (w_\infty\! +\! \D_\infty) \!\times \!\{-1\}\big) \;\subset\;N_\QQ \oplus \QQ$$
together with the subtorus action given by the lattice embedding
$N \hookrightarrow N \oplus \ZZ$. Here, we allow
$\D_0=\emptyset$ or $\D_\infty=\emptyset$.
Different choices of $w_0$ and $w_\infty$ lead to cones which can be transformed
into each other by a lattice automorphism of $N \times \ZZ$. Hence, the corresponding toric
varieties are isomorphic and the above statement makes indeed sense.
If the affine toric variety is assumed to be smooth, the cone has to be regular.
In this case, if $\D_0$ or $\D_\infty$ has dimension $n$, then the constructed cone has dimension $n+1$ and the variety $X(\D)$ is an affine space.
\end{remark}

It is not hard to exhibit the extremal rays of the cone constructed in Remark~\ref{rem:downgrade}.
\begin{lemma}
\label{sec:lemma-cone-rays}
There are three types of extremal rays in $C:=\cone(w_0+\D_0,w_\infty + \D_\infty)$:
\begin{enumerate}
\item $\rho \times \{0\}$ for every $\rho \in \sigma(1)$, where
$\deg \D \cap \rho = (w_0 + w_\infty + \D_0 + \D_\infty) \cap \rho = \emptyset$;
\item $\QQ_{\geq 0} \cdot (w_0+v, 1)$, where $v \in \D_0$ is a vertex;
\item $\QQ_{\geq 0} \cdot (w_\infty+v, -1)$, where $v \in \D_\infty$ is a vertex.
\end{enumerate}
\end{lemma}

\begin{proposition}{\cite[Proposition 3.1 and Theorem 3.3.]{h08}}
\label{sec:lemma-smooth}
  Let $\D$ be a p-divisor on $\PP^1$. Then $X(\D)$ is smooth if and only if
  \begin{enumerate}
  \item either $\deg \D \neq \emptyset$, $\D$ is of the form (\ref{eq:1}), and
  the cone $C$ is regular, or
  %hence $X(\D)$ is an affine space, or
  \item $\deg \D = \emptyset$ and $\cone(\D_y):=\cone(\D_y,\emptyset)$ is regular
  for every $y \in \PP^1$.
  \end{enumerate}
\end{proposition}

Polyhedral divisors of the second type do not necessarily correspond to affine spaces.
This is only the case if at most two coefficients are not lattice translates of the tail cone, see  Remark~\ref{rem:downgrade}.

As a consequence of Lemma~\ref{sec:lemma-cone-rays} and Proposition~\ref{sec:lemma-smooth}
we easily obtain that for two special cases all coefficients of $\D$ have to be translated
cones in order to obtain a smooth affine variety.
\begin{corollary}
\label{sec:cor-facet-rays}
Assume that $X(\D)$ is smooth. If $\D$ has a tail cone $\sigma$ of maximal dimension and
$\deg \D \cap \tau = \emptyset$ for some facet $\tau \prec \sigma$, then all the coefficients
are translates of $\sigma$ and all but two are even lattice translates.
\end{corollary}

\begin{corollary}
\label{sec:cor-smooth}
If $\deg \D = \emptyset$ and $X(\D)$ is smooth, then the tail cone $\sigma$ has to be regular.
Moreover, if $\sigma$ is maximal, then  $\D_y$ is either empty or a lattice translate of
$\sigma$ for every $y \in \PP^1$.
\end{corollary}

\subsection*{The complete case and affine coverings}
Consider two p-divisors $\D$ and $\D'$ on $\PP^1$ such that $\D_y'$ is a face of $\D_y$ for every $y \in \PP^1$ and $\deg \D' = \deg \D \cap \tail \D'$. Then by \cite[Proposition~1.1]{is}
we obtain an open embedding $X(\D') \hookrightarrow X(\D)$. For two p-divisors $\D$, $\D'$ we
define their \emph{intersection} by $\D\cap \D':= \sum_y (\D_y \cap \D'_y) \cdot y.$

For a given complete T-variety we consider an open covering by affine torus invariant subsets $X_i$, $i=1,\ldots,m$, and let $X_{ij}=X_i \cap X_j$. Every such subset corresponds to a polyhedral divisor $\D^i$ or $\D^{ij}$, respectively. We obtain a finite set $\fan = \{\D^1, \ldots, \D^m\}$. By \cite[Theorem 5.6]{ahs} and \cite[Remark 7.4(iv)]{ahs} we may assume that $\D^{ij}=\D^i \cap \D^j$ holds and the set $\fan$ satisfies the following compatibility conditions.
\begin{description}
%\smallskip
\item[Slice rule] The \emph{slices} $\fan_y=\{\D_y \mid \D \in \fan\}$ are complete polyhedral
subdivisions of $N_\QQ$, i.e. they cover $N_\QQ$ and the intersection of every two polyhedra is a face of both of them.
%\smallskip
\item[Degree rule] For $\tau = (\tail \D) \cap (\tail \D')$ one has
$\tau \cap (\deg \D)= \tau \cap (\deg \D').$
%\smallskip
\end{description}
Note that $\tail \fan := \{\tail \D \mid \D \in \fan\}$ generates a fan and all but finitely
many slices $\fan_y$ just equal $\tail \fan$. Consider a maximal tail cone $\sigma$ in
$\tail \fan$. Then for every $y$ there is a unique polyhedron $\fan_y(\sigma)$ in $\fan_y$
having this tail.

A maximal cone $\sigma \in \tail \fan$ is called \emph{marked} if the corresponding polyhedral
divisor $\D$ with $\sigma = \tail \D$ fulfills $\deg \D \neq \emptyset$. We denote the set
of all marked cones by $\tail^m(\fan) \subset \tail(\fan)$.

In general, there are many torus invariant affine coverings of $X$. But by
\cite[Proposition~1.6]{is} every rational complete $T$-variety of complexity one is uniquely
determined by the slices $\fan_y$ and the markings in $\tail \fan$. Hence, another set $\fan'$
of p-divisors with $\fan_y=\fan'_y$ for all $y \in \PP^1$ and $\tail^m(\fan)=\tail^m(\fan')$
corresponds to another invariant affine covering of the same variety.
From now on we assume
that $X$ is  a rational complete \emph{smooth} T-variety of complexity one and we consider
an affine covering given by the p-divisors in $\fan$. By Proposition~\ref{sec:lemma-smooth}, we have
\begin{lemma}
\label{sec:lemma-smooth-complete}
Given a maximal cone $\sigma$ in $\tail \fan$, there are two possible cases:
  \begin{enumerate}
  \item $\sigma$ is marked and all but two coefficients of $\fan_y(\sigma)$
  are lattice translates of $\sigma$, or
  \item $\sigma$ is not marked; then it has to be regular and $\fan_y(\sigma)$ has to be
  a lattice translate of $\sigma$ for every $y \in \PP^1$.
  \end{enumerate}
\end{lemma}

In the slices $\fan_y$ there might occur maximal polyhedra with non-maximal tail cones.
Here, Lemma~\ref{sec:lemma-smooth-complete} does not apply. Instead we need the following crucial fact.

\begin{proposition} \label{sec:prop-non-maximal}
Let $P$ be a maximal polyhedron with non-maximal tail in $\fan_z$ for some $z \in \PP^1$. Then up to one exception $z' \in \PP^1$ there is a lattice translate of $\tail(P)$ in $\fan_y$, for every $y \neq z$.
\end{proposition}

\begin{proof}
We denote the tail cone of $P$ by $\tau$. Consider the part $R$ of $\fan_{z}$ consisting of all maximal polyhedra with tail $\tau$. We are looking at the boundary facets of this part. There is
a facet having tail $\tau$, it corresponds to a primitive lattice element
$u \in \tau^\perp$, which is minimized on this facet. On the other side of the facet we have
a neighboring full-dimensional polyhedron $P'$ having a tail cone $\tau' \succ \tau$.
Replacing $P$ by $P'$ and iterating this procedure, we end up with a maximal polyhedron $P$,
a non-maximal tail cone $\tau = \tail P$, a region $R$ of $\fan_{z}$, and a facet of $R$
minimizing some $u \in \tau^\perp$ (which necessarily has tail cone $\tau$) such that
the neighboring polyhedron $\Delta^+$ has full-dimensional tail $\sigma^+$.
Now, we treat two cases separately:
\textbf{1)} $\dim \tau < n - 1$ and \textbf{2)} $\dim \tau = n - 1$.

In the first case, the common facet of $\Delta^+$ and $R$ has dimension $n - 1$, but tail
cone $\tau$ of dimension less than $n - 1$. This implies that the facet and, hence, $\Delta^+$
has at least $n - \dim \tau > 1$ vertices. In particular, it is not a lattice translate of
a cone and by Lemma~\ref{sec:lemma-smooth-complete} the tail cone $\sigma^+$ has to be marked.
Again by Lemma~\ref{sec:lemma-smooth-complete} for $y \neq z$ all but one of the
$\fan_y(\sigma^+)$ are lattice translate of $\sigma^+$. Hence, the faces of these
$\fan_y(\sigma^+)$ with tail cone $\tau$ are indeed lattice translates of $\tau$ and
the claim is proved.

In the second case, $-u$ is minimized on another facet of $R$. For the neighboring full-dimensional
polyhedron $\Delta^-$ we have $\tau \prec \sigma^-:=\tail \Delta^-$. Since $\tau$ is of
dimension $n -1$, the cone $\sigma^-$ must be full-dimensional. By construction
$\sigma^+ \cap \sigma^- = \tau$. Assume that $\sigma^+$ is not marked. Then all polyhedra
$\fan_y(\sigma^+) \in \fan_y$ are lattice translations of $\sigma^+$. As before, we infer
that the claim is fulfilled in this case. The same applies if  $\sigma^-$ is not marked.

Now assume that both $\sigma^+$ and $\sigma^-$ are marked. There are p-divisors
in $\D^+,\D^- \in \fan$ with $\tail \D^\pm  = \sigma^\pm$ and $\deg \D^\pm \neq \emptyset$.
If $\Delta^\pm=\D^\pm_z$ is not a lattice translate, then we know that all other polyhedra
$\D^\pm_y$ are lattice translates of $\sigma^\pm$ up to one exception. Hence, every $\D^\pm_y$
up to one exception contains a lattice translation of $\tau \prec \sigma^\pm$ and the claim
follows. Hence, we may assume that $\D^+_z$, $\D^-_z$ are just lattice translates
of the cone $\sigma^+$ and $\sigma^-$ respectively.

Remember that we have a maximal polyhedron $P \in \fan_z$ with non-maximal tail cone $\tau$.
Hence, there is some p-divisor $\D(P) \in \fan$ with $\D(P)_z=P$. By the properness
condition~(\ref{eq:properness}) we have $\deg \D(P)=\emptyset$ and by the degree rule
we have $\tau \cap \deg \D^\pm = \emptyset$. Now, by Corollary~\ref{sec:cor-facet-rays} we know
that all $\D^\pm_y$ are just translated cones $(v_y^\pm + \sigma^\pm)$. Moreover, up to
two exceptions $\D_{y_0}^\pm=(v^\pm_0+\sigma)$ and $\D_{y_\infty}^\pm = (v_\infty^\pm + \sigma)$
they are even lattice translates, i.e. $v^\pm_y \in N$.

Corollary~\ref{sec:cor-smooth} ensures that $\tau$ is a regular cone. Hence, the primitive
ray generators $e_1, \ldots e_{n-1}$ of $\tau$ form a part of a basis $e_1, \ldots e_{n}$
of $N$. Since $u \in \tau^\perp$ we have $\langle u, e_n \rangle=1$. Now, the elements $(e_i,0)$
together with $(0,1)$ form a basis of $N \times \ZZ$. We use this basis for an identification
$N \times \ZZ \cong \ZZ^{n+1}$. In particular, $\langle u, \cdot \rangle$ equals to the
$n$-th coordinate in this basis.

By Lemma~\ref{sec:lemma-cone-rays}, the primitive ray generators of
$\cone(w^\pm_0+\D_{y_0}^\pm, w^\pm_\infty +\D_{y_\infty}^\pm)$
(as in Remark~\ref{rem:downgrade}) are given by the columns of the following matrix. Due to the smoothness condition these matrices have to be unimodular. There first $n-1$ columns correspond to the rays of $\tau$ and the last two columns to
the vertex in $\D_{y_0}$ and $\D_{y_\infty}$, respectively.

Here, $\mu_0^\pm$, $\mu_\infty^\pm$ are minimal positive integers such that
$\mu^\pm_0 \cdot v^\pm_0$ and $\mu^\pm_\infty \cdot v^\pm_\infty$
are lattice elements.
By the slice rule, we have $\langle u, v^+_y \rangle \geq \langle u, v^-_y \rangle$
(else $(v^+_y + \sigma^+)$ and $(v^-_y + \sigma^-)$ would intersect in a non-face, since
$\tau = \sigma^+ \cap u^\perp = \sigma^-\cap u^\perp$ is a common facet). Moreover,
%\begin{equation}
 $ \langle v^+_z, u \rangle  >  \langle v^-_z, u \rangle$
%\end{equation}
holds,  since $\Delta^+=(v_z^+ +\sigma^+)$ and $\Delta^-=(v_z^- +\sigma^-)$ are separated by the full-dimensional region~$R$. Note that the compared values are integers. Let us set
$\Sigma^\pm=\sum_y v^\pm_y$. By definition, we have $v^\pm = \Sigma^\pm-v_0^\pm-v^\pm_\infty$.
We obtain
%\begin{equation}
%  \label{eq:ineq-separated}
$  \langle \Sigma^+, u \rangle  \geq  \langle \Sigma^-, u \rangle + 1.$
%\end{equation}

\[M^\pm=
\left(\begin{array}{ccccc}
  1&      & &* &*\\
   &\ddots& &\vdots &\vdots \\
   &      &1&* &* \\
  0&\cdots&0&\langle v^\pm_0 +w^\pm_0, u \rangle & \langle v^\pm_\infty +w^\pm_\infty, u \rangle\\
  0&\cdots&0&\mu^\pm_0&-\mu^\pm_\infty\\
\end{array}\right)
\]

We choose $w_0^+$ in a way such that $0 \leq \langle v^+_0 + w^+_0, u \rangle < 1$ holds and
set $w_\infty^+ = v^+ - w^+_0$,
$w_\infty^- =  w^+_\infty - \lfloor v^-_\infty-v^+_\infty \rfloor$ (componentwise rounding)
and $w_0^-=v^--w_\infty^-$. Hence, we obtain
$\langle v^-_\infty + w^-_\infty, u \rangle \leq  \langle v^+_\infty + w^+_\infty, u  \rangle$
and
\begin{align*}
  v^-_0+w^-_0 &= v^-_0 + v^--w_\infty^- = \Sigma^- - v^-_\infty - w_\infty^-\\
&= \Sigma^- - v^-_\infty - w^+_\infty + \lfloor v_\infty^--v^+_\infty \rfloor\\
&= \Sigma^- - v^-_\infty -  v^+ + w^+_0 + \lfloor v_\infty^--v^+_\infty \rfloor\\
&= \Sigma^- - v^-_\infty - \Sigma^+ + v^+_0+ v^+_\infty  + w^+_0 + \lfloor v^-_\infty-v^+_\infty \rfloor \\
&= w^+_0+ v_0^+ + (\Sigma^- - \Sigma^+) + \big(\lfloor v^-_\infty-v^+_\infty \rfloor - (v^-_\infty-v^+_\infty)\big).
\end{align*}
After pairing with $u$ we obtain
$\langle v_0^-+w_0^-, u \rangle \leq \langle w^+_0+v^+_0,u \rangle - 1 < 0$.
Hence, either $\langle v^+_0 +w^+_0, u \rangle,\; \langle v^+_\infty +w^+_\infty, u \rangle \geq 0$
or $\langle v^-_0 +w^-_0, u \rangle,\; \langle v^-_\infty +w^-_\infty, u \rangle \leq 0$.
In both cases we need to have either $\mu_0^\pm=1$ or $\mu^\pm_\infty=1$ in order to obtain $|\det M^\pm|= 1$.
All but one coefficient of $\D^+$ or $\D^-$, respectively, are lattice translates.
Since $\tau$ is a face of $\sigma^\pm$ we will always find a lattice translate of $\tau$
as well, and Proposition~\ref{sec:prop-non-maximal} is proved.
\end{proof}

\begin{proof}[Proof of Theorem~\ref{sec:thm-cplx-one}]
Consider a set $\fan$ of p-divisors giving rise to a covering of $X$ as above.
We construct another set of p-divisors $\fan'$ giving rise to an $A$-covering of $X$.

Let $\sigma$ be a marked maximal cone in $\tail \fan$. There is a $\D \in \fan$
with $\deg \D \neq \emptyset$ and $\tail \D = \sigma$. We simply add it to $\fan'$.
By Lemma~\ref{sec:lemma-smooth}, $X(\D)$ is an affine space. If $\sigma$ is maximal
but not marked, then by Lemma~\ref{sec:lemma-smooth-complete} the polyhedra $\fan_y(\sigma)$
are just lattice translates of $\sigma$. Now, we add the following two polyhedral divisors
to $\fan'$:
\[\D_0 = \emptyset \cdot 0 + \sum_{y \neq 0} \fan_y(\sigma) \cdot y \quad \text{and} \quad
\D_\infty =\emptyset \cdot \infty + \sum_{y \neq \infty} \fan_y(\sigma)\cdot y.\]
From Remark~\ref{rem:downgrade} we know that $X(\D_0)$ and $X(\D_\infty)$ are both
affine spaces.

By these considerations $\fan'_y$ covers all polyhedra from $\fan_y$ having maximal tail cones.
Moreover, the markings are the same as for $\fan$. It remains to care
for maximal polyhedra $P$ having non-maximal tail $\tau$. We consider such a polyhedron living in some slice $\fan_{z}$. By Proposition~\ref{sec:prop-non-maximal}, we have a lattice translate
$(v_y + \tau)$  in every slice except for $\fan_z$ and $\fan_{z'}$. Having this, we can add the p-divisor $\D(P) = \emptyset \cdot z' + P \cdot z + \sum_{y \neq z,z'} (v_y + \tau) \cdot y$
to $\fan'$. Thus for all maximal polyhedra with non-maximal tail
we obtain $\fan_y = \fan'_y$ for all $y \in \PP^1$.
From Remark~\ref{rem:downgrade} we know that $X(\D(P))$
are affine spaces. Hence, we obtain an $A$-covering of~$X$.
\end{proof}

\begin{remark}
By Remark~\ref{rem:non-a-covered}, for complexity three Theorem~\ref{sec:thm-cplx-one}
does not hold. For complexity two, Theorem~\ref{sec:thm-cplx-one} holds at least for surfaces and the threefolds $V_5$ and $V_{22}'$ carrying a $\KK^{\times}$-action.
\end{remark}

%%%%%%%%%%%%%%%%%%%%%%%%%%%%%%%%%%%%%%%%%%%%%%%%%%%%%%%%%%%%%%%

\section*{Acknowledgement}

The authors are grateful to Mikhail Zaidenberg for useful comments and remarks.
Special thanks are due to the referee for careful reading, constructive criticism
and important suggestions.

%%%%%%%%%%%%%%%%%%%%%%%%%%%%%%%%%%%%%%%%%%

%
\end{document}